\documentclass[a4paper, 11pt]{article}

\usepackage[
            includefoot,  
            marginparwidth=0in,     
            marginparsep=0in,       
            margin=1.45in,               
            includemp]{geometry}

\usepackage{bbm}
\usepackage{float}
\usepackage{graphicx}                  
\usepackage{amssymb}
\usepackage{amsfonts}
\RequirePackage{amsmath}
\RequirePackage{amsthm}

\usepackage{color}

\usepackage{fancyhdr}

\newtheorem{thm}{Theorem}[section]
\newtheorem{defn}[thm]{Definition}
\newtheorem{prop}[thm]{Proposition}

\newtheorem{cor}[thm]{Corollary}

\newtheorem{rmq}{Remark}[section]

\DeclareMathOperator{\Tr}{Tr}
\newcommand{\ev}{\operatorname{ev}}

\newcommand{\N}{\mathbb{N}}

\begin{document}

\title{A note on existence of free Stein kernels}
\author{Guillaume C\'ebron\thanks{Institut de Math\'ematiques de Toulouse, Universit\'e de Toulouse, guillaume.cebron@math.univ-toulouse.fr}, \hspace{1mm} Max Fathi\thanks{CNRS and Institut de Math\'ematiques de Toulouse, Universit\'e de Toulouse, max.fathi@math.univ-toulouse.fr} \hspace{1mm} and Tobias Mai\thanks{ Faculty of Mathematics, Saarland University, mai@math.uni-sb.de}}
\date{\today}

\maketitle

\begin{abstract}
Stein kernels are a way of comparing probability distributions, defined via integration by parts formulas. We provide two constructions of Stein kernels in free probability. One is given by an explicit formula, and the other via free Poincar\'e inequalities. In particular, we show that unlike in the classical setting, free Stein kernels always exist. As corollaries, we derive new bounds on the rate of convergence in the free CLT, and a strengthening of a characterization of the semicircular law due to Biane. 
\end{abstract}

MSC: 46L54; 60F05.

\section{Introduction}

Let $M$ be a von Neumann algebra and $\varphi$ a faithful normal state. Let $\mathcal{P} = \mathbb{C}\langle t_1,\dots,t_n \rangle$ be the algebra of noncommutative polynomials in $n$ variables $t_1,\dots,t_n$.

We are interested in Stein kernels, which were defined in \cite{FN17} as follows:  

\begin{defn}
A free Stein kernel for a $n$-tuple $X$, with respect to a potential $V \in \mathcal{P}$, is an element $A \in L^2(M_n(M \bar{\otimes} M^{op}, (\varphi \otimes \varphi^{op})\circ \Tr))$ such that for any $P \in \mathcal{P}^n$ we have
$$\langle [\mathcal{D}V](X), P(X)\rangle_{\varphi} = \langle A, [\mathcal{J}P](X) \rangle_{(\varphi \otimes \varphi^{op})\circ \Tr}.$$
The free Stein discrepancy of $X$ relative to $V$ is defined as
$$\Sigma^*(X|V) := \inf_A \|A - (1 \otimes 1^{op})\otimes I_n\|_{L^2(M_n(M \bar{\otimes} M^{op}, (\varphi \otimes \varphi^{op})\circ \Tr))}$$
where the infimum is taken over all admissible Stein kernels $A$ of $X$. 
\end{defn}

In this definition, $\mathcal{D}V$ is the cyclic gradient of the potential $V$, and $\mathcal{J}P$ is the Jacobian matrix of $P$. We shall recall the precise definitions of these notions in Section 2.

The motivation behind this definition is that $(1 \otimes 1^{op})\otimes I_n$ is a Stein kernel for $X$ relative to $V$ if and only if $X$ is a free Gibbs state with potential $V$. Hence these kernels and the discrepancy give an estimate of how far away $X$ is from being this Gibbs state. When $V$ is the quadratic potential $\frac{1}{2}\sum_{i=1}^n t_i^2$, this corresponds to the situation where the reference state is an $n$-dimensional semicircular law. Stein kernels have been used to establish functional inequalities, and to estimate rates of convergence in approximation theorems, such as the free central limit theorem. Their classical counterpart are one way of implementing the now widely used Stein's method \cite{Ste72, Ros11} to bound distances between probability distributions. Another way of implementing Stein's method that has recently been adapted to the context of free probability is to use it in combination with Malliavin calculus \cite{KNPS}.

\begin{rmq}
Instead of working with polynomials, we could consider power series $\mathcal{P}^{(R)}$, as in \cite{FN17}. All the results proven below remain valid in that generality, by approximation with polynomials. 
\end{rmq}

An immediate constraint is that for a Stein kernel to exist, a necessary condition is that $\varphi([\mathcal{D}V](X)) = (0,\dots,0)$, since otherwise we would get a contradiction by testing the above identity with constant polynomials. We shall show that this is also a sufficient condition: 

\begin{thm}
Given any potential $V\in \mathcal{P}$ such that $\varphi([\mathcal{D}V](X)) = (0,\dots,0)$ and any self-adjoint $n$-tuple $X = (x_1,\dots,x_n) \in M^n$,  there exists a Stein kernel for $X$ relative to $V$. 
\end{thm}

This result stands in contrast with the situation in the classical setting, where Stein kernels may not exist, and known constructions typically require some regularity condition \cite{Cha09, CFP17, NP12}.

We shall give two constructions that work in full generality: one explicit construction, with an exact formula, and one implicit, via the Riesz representation theorem. Interestingly, the two constructions do \emph{not} give the same kernels, which in particular implies that free Stein kernels are never unique. In the classical setting, Stein kernels are not generally unique in dimension higher than two, but are unique for one-dimensional measures with nice density. 

An immediate consequence of our constructions will be the following rate of convergence in the free CLT, obtained via Theorem 5.2 of \cite{FN17}: 

\begin{thm}
Let $(X^{(i)})_{i \in \N}$ be a sequence of freely independent, identically distributed $n$-tuples of self-adjoint random variables. Assume that each of them are centered, and that the covariance matrix of $X = (x_1,\dots,x_n)$ is the identity. Let $Y^k = k^{-1/2}\left(\sum_{i=1}^{k} X^{(i)}\right)$. We have
$$\Sigma^*\left(Y^k\left|\frac{1}{2}\sum_{i=1}^n t_i^2 \right.\right) \leq \frac{C}{\sqrt{k}}$$
with $C=\sqrt{n \min\left(C_{opt}(X)-1,(n+nm_4-1)/2\right)}$ (where $m_4$ is the maximum fourth moment $\max_{i=1}^n\varphi(x_i^4)$ and $C_{opt}(X)$ is the Poincar\'e constant of $X$). If moreover the non-microstates free Fisher information of $X$ is finite, then the non-microstates free relative entropy $\chi^*$ satisfies
$$\chi^*\left(Y^k\left|\frac{1}{2}\sum_{i=1}^n t_i^2 \right.\right) \leq \frac{C\log k}{k}$$
with a constant $C$ that only depends on the free Fisher information and the norm of $X$. 
\end{thm}

As an immediate corollary obtained via a comparison to the Stein discrepancy proved in \cite[Section 2.5]{C18}, we obtain the rate of convergence $W_2(Y^k,S)\leq C \cdot k^{-1/2}$ for the non-commutative $2$-Wasserstein distance of \cite{BV01} from $Y^k$ to a $n$-tuple $S$ of free standard semicircular variables. Interestingly, obtaining such rates in classical probability under boundedness or moment assumptions is much more difficult, since Stein kernels may not exist \cite{Bon16, Zha17}. For one-dimensional free probability, a sharper rate, without the logarithmic factor, has been obtained in \cite{CG13}. A rate of convergence at the level of the Cauchy transforms of polynomial evaluations was obtained in \cite{MS13}; see \cite{BM18} for an improved result.

\section{Explicit formula}

In this section, we want to prove existence of free Stein kernels for an arbitrary potential by providing an explicit formula. This requires some notation.

The unital complex algebra $\mathcal{P} = \mathbb{C}\langle t_1,\dots,t_n\rangle$ becomes a $\ast$-algebra if it is endowed with the involution $\ast:\mathcal{P}\to \mathcal{P}$, that is the unital antilinear map uniquely determined by $(PQ)^*=Q^*P^*$ and $t_i^*=t_i$ for $i=1,\dots,n$.
This map naturally induces an involution $\ast: \mathcal{P} \otimes \mathcal{P} \to \mathcal{P} \otimes \mathcal{P}$ on the algebraic tensor product $\mathcal{P}\otimes \mathcal{P}$ over $\mathbb{C}$ by $(P\otimes Q)^* = P^*\otimes Q^*$.

Furthermore, we define on $\mathcal{P} \otimes \mathcal{P}$ the binary operation
$$\sharp:\ (\mathcal{P} \otimes \mathcal{P})^2 \to \mathcal{P} \otimes \mathcal{P},\qquad (P_1 \otimes P_2) \sharp (Q_1 \otimes Q_2) = (P_1Q_1) \otimes (Q_2P_2),$$
as well as the \emph{multiplication mapping}
$$m:\ \mathcal{P} \otimes \mathcal{P} \to \mathcal{P}, \qquad P \otimes Q \mapsto PQ,$$
and the \emph{flip mapping}
$$\sigma:\ \mathcal{P} \otimes \mathcal{P} \to \mathcal{P} \otimes \mathcal{P},\qquad P\otimes Q \mapsto Q\otimes P.$$

We may view $\mathcal{P}\otimes \mathcal{P}$ as an $\mathcal{P}$-$\mathcal{P}$-bimodule via the natural actions that are determined by $P_1\cdot (Q_1\otimes Q_2)\cdot P_2= (P_1 Q_1)\otimes (Q_2 P_2)$. Accordingly, we can consider derivations $d$ on $\mathcal{P}$ with values in $\mathcal{P}\otimes \mathcal{P}$, i.e., linear maps $d: \mathcal{P} \to \mathcal{P} \otimes \mathcal{P}$ that satisfy the \emph{Leibniz rule}
$$d(PQ) = d(P) \cdot Q + P \cdot d(Q) \qquad\text{for all $P,Q\in\mathcal{P}$}.$$
Based on this terminology, we may introduce the \emph{non-commutative derivatives} $\partial_1,\dots,\partial_n$ as the unique derivations
$$\partial_i:\ \mathcal{P} \to \mathcal{P} \otimes \mathcal{P},\qquad i=1,\dots,n,$$
that satisfy $\partial_i(t_j) = \delta_{ij} 1 \otimes 1$ for $j=1,\dots,n$. Further, we will work with the linear map
$$\delta:\ \mathcal{P} \to \mathcal{P} \otimes \mathcal{P}, \qquad P \mapsto P\otimes 1 - 1 \otimes P,$$
which is easily checked to be a derivation as well. Note that we can write $\delta(P)=[P,1\otimes 1]$ with respect to the $\mathcal{P}$-$\mathcal{P}$-bimodule structure of $\mathcal{P}\otimes \mathcal{P}$.
From their definitions, one easily infers that those derivations are related by the formula
\begin{equation}\label{eq:Voi-formula}
\delta(P) = \sum_{i=1}^n \partial_i(P) \sharp \delta(t_i) \qquad\text{for all $P\in\mathcal{P}$}.
\end{equation}
We point out that this formula underlies the proof of Proposition \ref{prop:Poincare} below.

For any given $P=(P_1,\dots,P_n) \in \mathcal{P}^n$, the \emph{Jacobian matrix} $\mathcal{J}P$ of $P$ is defined by
$$\mathcal{J}P=\left(\begin{array}{ccc}
\partial_1 P_1 & \cdots& \partial_n P_1  \\ 
\vdots & \ddots & \vdots \\ 
\partial_1 P_n  & \cdots & \partial_n P_n 
\end{array} \right) \in M_n(\mathcal{P} \otimes \mathcal{P}).$$

Associated to the noncommutative derivatives are the so-called \emph{cyclic derivatives}; those are the linear maps
$$\mathcal D_i:\ \mathcal{P} \to \mathcal{P},\qquad i=1,\dots,n,$$
that are defined by $\mathcal D_i := m\circ \sigma \circ \partial_i$. For any given potential $V \in \mathcal{P}$, we denote by $\mathcal{D}V$ the \emph{cyclic gradient} of $V$ which is defined as
$$\mathcal{D}V = (\mathcal{D}_1V, \ldots, \mathcal{D}_nV) \in \mathcal{P}^n $$

For every $n$-tuple $X=(x_1,\dots,x_n)$ of self-adjoint operators in the von Neumann algebra $M$, we have a canonical evaluation homomorphism $\ev_X: \mathcal{P} \to M$ that is determined by $\ev_X(1)=1$ and $\ev_X(t_i) = x_i$ for $i=1,\dots,n$; we put $P(X) := \ev_X(P)$.
Analogously, we define $Q(X) := (\ev_X \otimes \ev_X)(Q)$ for every $Q\in \mathcal{P} \otimes \mathcal{P}$. Even though by definition $Q(X) \in M\otimes M$, we will often view $Q(X)$ as an element in $M \otimes M^{op}$; note that then $(Q_1 \sharp Q_2)(X) = Q_1(X) \cdot Q_2(X)$ for all $Q_1,Q_2\in \mathcal{P} \otimes \mathcal{P}$.
Similarly, we can evaluate elements in $\mathcal{P}^n$ and $M_n(\mathcal{P} \otimes \mathcal{P})$.

On the complex vector spaces $M^n$ and $M_n(M \otimes M^{op})$, we may introduce the inner products $\langle \cdot,\cdot \rangle_{\varphi}$ and $\langle \cdot,\cdot\rangle_{(\varphi \otimes \varphi^{op})\circ \Tr}$ by
$$\langle X,Y\rangle_{\varphi}=\sum_{i=1}^n\varphi(x_i y_i^*)$$
for all $X=(x_1,\ldots, x_n), Y=(y_1,\ldots, y_n) \in M^n$ and by
$$\langle A,B\rangle_{(\varphi \otimes \varphi^{op})\circ \Tr}=\sum_{i,j=1}^n(\varphi \otimes \varphi^{op})(a_{i,j} \cdot b_{i,j}^*)$$
for all $A=(a_{i,j}),B=(b_{i,j})\in M_n(M \otimes M^{op})$, respectively. Note that
$$\langle A,B\rangle_{(\varphi \otimes \varphi^{op})\circ \Tr}=(\varphi \otimes \varphi^{op})\circ \Tr(A \cdot B^*)$$
where $B^*$ is given by $B^*=(b_{j,i}^*)_{i,j=1}^n$.

Now, we are ready to formulate the announced theorem which gives a free Stein kernel by an explicit formula. This is inspired by an unpublished note of Giovanni Peccati and Roland Speicher, where $\frac{1}{2}(X^2 \otimes 1 + 1 \otimes X^2 - 2 X \otimes X)$ was given as an expression for a free Stein kernel in the single variable case $n=1$ and for the special potential $V=\frac{1}{2}t^2$.

\begin{thm}
Given any $V\in \mathcal{P}$, the matrix
$$A :=\frac{1}{2}\Big(\delta(\mathcal{D}_iV)\sharp [\delta(t_j)]\Big)_{i,j=1}^n\in M_n(\mathcal{P} \otimes \mathcal{P})$$
satisfies
$$\langle [\mathcal{D}V](X)-\varphi([\mathcal{D}V](X)), P(X)\rangle_{\varphi} = \langle A(X), [\mathcal{J}P](X) \rangle_{(\varphi \otimes \varphi^{op})\circ \Tr}.$$
In particular, if $\varphi([\mathcal{D}V](X)) = (0,\ldots,0)$, then $A(X)\in M_n(M \otimes M^{op})$ is a free Stein kernel for $X$ relative to $V$.

Moreover, if $V=\frac{1}{2}\sum_{i=1}^n t_i^2$ and $X$ is centered, then
$$\Sigma^*(X|V)^2 \leq \|A - (1 \otimes 1^{op})\otimes I_n\|_{L^2(M_n(M \bar{\otimes} M^{op}, (\varphi \otimes \varphi^{op})\circ \Tr))}\leq (n^2+n^2 m_4-n)/2. $$
\end{thm}

\begin{proof}
Let $P\in \mathcal{P}$ be given. Remark that, in $M_n(M \otimes M^{op})$,
$$A(X)\cdot ([\mathcal{J}P](X))^*=[A\sharp (\mathcal{J}P)^*](X),$$
where $A\sharp (\mathcal{J}P)^*$ is the element of $M_n(\mathcal{P} \otimes\mathcal{P})$ whose $(i,j)$-coefficient is
\begin{align*}
[A\sharp (\mathcal{J}P)^*)]_{i,j}&=\sum_{k=1}^n\frac{1}{2}\Big(\delta(\mathcal{D}_iV)\sharp [\delta(t_k)]\Big)\sharp \Big((\partial_k P_j)^*\Big)\\
&=\sum_{k=1}^n\frac{1}{2}\delta(\mathcal{D}_iV)\sharp \Big( [\delta(t_k)]^*\sharp (\partial_k P_j)^*\Big) =\sum_{k=1}^n\frac{1}{2}\delta(\mathcal{D}_iV)\sharp \Big( \big( (\partial_k P_j) \sharp [\delta(t_k)]\big)^*\Big)\\
&=\frac{1}{2}\delta(\mathcal{D}_iV)\sharp \bigg(\Big( \sum_{k=1}^n (\partial_k P_j) \sharp [\delta(t_k)] \Big)^*\bigg) \stackrel{\eqref{eq:Voi-formula}}{=} \frac{1}{2}\delta(\mathcal{D}_iV)\sharp \Big( \delta(P_j)^*\Big).
\end{align*}
If $i=j$, we can pursue the computation of the diagonal $(i,i)$-coefficient, namely
\begin{align*}
[A\sharp (\mathcal{J}P)^*)]_{i,i}&=\frac{1}{2}(\mathcal{D}_iV\otimes 1-1\otimes\mathcal{D}_iV)\sharp ( P_i^*\otimes 1-1\otimes P_i^* )\\
&=\frac{1}{2}\Big((\mathcal{D}_iV)P_i^*\otimes 1-P_i^*\otimes (\mathcal{D}_iV)-(\mathcal{D}_iV)\otimes P_i^* +1\otimes (\mathcal{D}_iV)P_i^*\Big),
\end{align*}
from which we deduce that
$$(\varphi \otimes \varphi^{op})\left([A\sharp (\mathcal{J}P)^*)]_{i,i}(X)\right)=\varphi\Big( \big([\mathcal{D}_i V](X)-\varphi([\mathcal{D}_i V](X)) \big) P_i^*(X) \Big).$$
Using this observation, we may now check that
\begin{align*}
\langle A(X), &[\mathcal{J}P](X) \rangle_{(\varphi \otimes \varphi^{op})\circ \Tr}=(\varphi \otimes \varphi^{op})\circ \Tr(A(X)\cdot ([\mathcal{J}P](X))^*)\\
&=\sum_{i=1}^n(\varphi \otimes \varphi^{op})\big([A\sharp (\mathcal{J}P)^*)]_{i,i}(X)\big)\\
&=\sum_{i=1}^n \varphi\Big( \big([\mathcal{D}_i V](X)-\varphi([\mathcal{D}_i V](X))\big) P_i^*(X)\Big)\\
&=\langle [\mathcal{D}V](X)-\varphi([\mathcal{D}V](X)), P(X)\rangle_{\varphi} 
\end{align*}
which is the asserted formula.
Now, if $V=\frac{1}{2}\sum_{i=1}^n t_i^2$ and $X$ is centered, we have
$$A(X)=\frac{1}{2}\Big((x_i\otimes 1-1\otimes x_i)\cdot (x_j\otimes 1-1\otimes x_j)\Big)_{i,j=1}^n$$
and consequently,
\begin{align*}&\|A - (1 \otimes 1^{op})\otimes I_n\|^2_{L^2(M_n(M \bar{\otimes} M^{op}, (\varphi \otimes \varphi^{op})\circ \Tr))}\\&=\|A\|^2_{L^2(M_n(M \bar{\otimes} M^{op}, (\varphi \otimes \varphi^{op})\circ \Tr))}-n\\
&=\frac{1}{4}\sum_{i,j=1}^n\Big(2\varphi(x_ix_j^2x_i)+2\varphi(x_ix_j)\varphi(x_jx_i)++2\varphi(x_j^2)\varphi(x_i^2)\Big)-n\\
&\leq \frac{1}{2}(n^2 m_4+n+n^2)-n=(n^2+n^2 m_4-n)/2,
\end{align*}
which concludes the proof. 
\end{proof}

\section{Construction via a free Poincar\'e inequality}

In this section, we adapt the construction of \cite{CFP17} to the free setting. To state our results, we first define free Poincar\'e inequalities: 

\begin{defn}
A self-adjoint $n$-tuple $X$ satisfies a free Poincar\'e inequality with constant $C$ is for any polynomial $P \in \mathcal{P}$ we have
$$\|P(X) - \varphi(P(X))\|_2^2 \leq C\underset{i=1}{\stackrel{n}{\sum}} \hspace{1mm} \|\partial_i P(X)\|_2^2.$$
The best possible constant in this inequality will be denoted by $C_{opt}(X)$. 
\end{defn}

In \cite{Bia03}, Biane showed that the best constant in the Poincar\'e inequality for a semicircular law is one, and that this property characterizes it among all $n$-tuples with covariance matrix equal to the identity. We shall later see a refinement of this result. The Poincar\'e constant is always finite, and satisfies the following bound: 

\begin{prop}\label{prop:Poincare}
Consider a self-adjoint $n$-tuple $X=(x_1,\dots,x_n)$. For any self-adjoint $P \in \mathcal{P}$ we have
$$\|P(X) - \varphi(P(X))\|_2^2 \leq 2n\|X\|^2\underset{i=1}{\stackrel{n}{\sum}} \hspace{1mm} \|\partial_i P(X)\|_2^2,$$
where $\|X\| := \max^n_{i=1} \|x_i\|$. If $P$ is not self-adjoint, we have
$$\|P(X) - \varphi(P(X))\|_2^2 \leq 4n\|X\|^2\underset{i=1}{\stackrel{n}{\sum}} \hspace{1mm} \|\partial_i P(X)\|_2^2,$$
so that $C_{opt}(X) \leq  4n\|X\|^2$. 
\end{prop}

The first part is a result of Voiculescu, which can be found in \cite{Dab10, MS17}. The second part is a trivial extension to non self-adjoint polynomials: if we consider a general $P$, it can be written as $P = P_1 + iP_2$. Then 
\begin{align*}
\|P(X) - \varphi(P(X))\|_2^2 &\leq 2\|P_1(X) - \varphi(P_1(X))\|_2^2 + 2\|P_2(X) - \varphi(P_2(X))\|_2^2 \\
&\leq 4n\|X\|^2\underset{j=1}{\stackrel{n}{\sum}} \hspace{1mm} \|\partial_j P_1(X)\|_2^2 + \|\partial_j P_2(X)\|_2^2 \\
&= 4n\|X\|^2\underset{j=1}{\stackrel{n}{\sum}} \hspace{1mm} \|\partial_j P(X)\|_2^2.
\end{align*}

We note that if $\varphi$ is tracial, then the bound on $C_{opt}(X)$ given in Proposition \ref{prop:Poincare} can be improved to $C_{opt}(X) \leq 2n\|X\|^2$, with the same proof based on \eqref{eq:Voi-formula} as in the self-adjoint case.

The construction we shall now see will lead to the following estimates:  

\begin{thm} \label{thm_stein_const2}
Assume that $\varphi([\mathcal{D}V](X)) = 0$. Then there exists a free Stein kernel $A$ for $X$ relative to $V$, and moreover it satisfies 
$$\Sigma^*(X|V)^2 \leq \|A - (1 \otimes 1^{op})\otimes I_n\|_{(\varphi \otimes \varphi^{op})\circ \Tr}^2 \leq n + C_{opt}(X)\|[\mathcal{D}V](X)\|_2^2 - 2\langle [\mathcal{D}V](X), X\rangle_{\varphi}.$$

In particular, if $X=(x_1,\dots,x_n)$ is centered and satisfies $\sum^n_{i=1} \varphi(x_i^2) = n$, then we have
$$\Sigma^*\left(X \left|\frac{1}{2}\sum_{i=1}^n t_i^2\right.\right)^2 \leq n(C_{opt}(X)-1).$$
\end{thm}

Before giving the actual construction, it will be convenient to introduce the free Sobolev space $H^1(\mu)$ associated to a noncommutative distribution $\mu: \mathcal{P} \to \mathbb{C}$. We can think of it as the distribution associated to a $n$-tuple of variables $\mu: \mathcal{P} \to \mathbb{C}, P \mapsto \varphi(P(X))$. We can define on $\mathcal{P}^n$ a sesqui-linear form $\langle \cdot,\cdot \rangle_{H^1(\mu)}$ by
$$\langle P, Q \rangle_{H^1(\mu)} := (\mu\otimes\mu)\circ \Tr\big((\mathcal{J}Q)^\ast \sharp (\mathcal{J}P)\big) \qquad\text{for $P,Q\in\mathcal{P}^n$}.$$
The latter then induces an inner product on the quotient space $\mathcal{P}^n / \mathcal{N}_\mu$, which we denote again by $\langle \cdot,\cdot\rangle_{H^1(\mu)}$, where $\mathcal{N}_\mu := \{ P\in \mathcal{P}^n \mid \langle P, P\rangle_{H^1(\mu)} = 0\}$. Unlike in the classical setting, there may be nonconstant tuples of polynomials in $\mathcal{N}_\mu$, since due to relations between the $x_i$ the Jacobian matrix $[\mathcal{J}P](X)$ may be zero. The Hilbert space obtained by completing $\mathcal{P}^n / \mathcal{N}_\mu$ with respect to the norm $\|\cdot\|_{H^1(\mu)}$ induced by $\langle \cdot,\cdot \rangle_{H^1(\mu)}$ is then denoted by $H^1(\mu)$.

\begin{rmq}
In view of the classical setting, it would be more natural to define the scalar product on $H^1(\mu)$ as $\mu(Q^\ast P) + (\mu\otimes\mu)\circ \Tr\big((\mathcal{J}Q)^\ast \sharp (\mathcal{J}P)\big)$. Because of the Poincar\'e inequality, the two norms would be equivalent on $\mathcal{P}^n / \mathcal{N}_\mu$, and for our purpose the definition we gave above is better adapted. 
\end{rmq}

\begin{proof}[Proof of Theorem \ref{thm_stein_const2}]
Let $\mu$ be the distribution associated to the $X_i$. The application 
$$f_{V,X}:\ P \mapsto \langle [\mathcal{D}V](X), P(X)\rangle_{\varphi}$$
is a linear form over $\mathcal{P}^n$. We shall show that there exists a constant $C>0$ such that
\begin{equation}\label{eq:functional-bound}
|f_{V,X}(P)| \leq C\|P\|_{H^1(\mu)} \qquad\text{for any $P \in \mathcal{P}^n$},
\end{equation}
where $\|\cdot\|_{H^1(\mu)}$ stands for the semi-norm induced by $\langle \cdot,\cdot \rangle_{H^1(\mu)}$ on $\mathcal{P}^n$. Once we will have proved this, we may infer from the definition of $\mathcal{N}_\mu$ that $f_{V,X}$ is invariant by adding an element of $\mathcal{N}_\mu$, so that we can view it as a linear form on the quotient space $H^1(\mu)$; in fact, we see that this yields a continuous form on $H^1(\mu)$. Thus, we shall be able to apply the Riesz representation theorem: there exists a unique element $Q_0$ of $H^1(\mu)$ such that for any element $P$ of $H^1(\mu)$ we have
$$f_{V,X}(P) = \langle Q_0, P \rangle_{H^1(\mu)}.$$
The map $Q\mapsto [\mathcal{J}Q](X)$ is an isometry from $\mathcal{P}^n / \mathcal{N}_\mu$ to $L^2(M_n(M \bar{\otimes} M^{op}, (\varphi \otimes \varphi^{op})\circ \Tr))$ which extends to $H^1(\mu)$, and denoting by $[\mathcal{J}Q_0](X)$ the image of $Q_0$ via this isomorphism, we have
$$f_{V,X}(P) = \langle [\mathcal{J}Q_0](X), [\mathcal{J}P](X) \rangle_{(\varphi \otimes \varphi^{op})\circ \Tr}$$
for any element $P$ of $\mathcal{P}^n$
and hence $[\mathcal{J}Q_0](X)$ would be a free Stein kernel for $X$, relative to $V$.

So all that is left to show is \eqref{eq:functional-bound}. As pointed out in \cite{CFP17}, this continuity can be proved when a Poincar\'e inequality holds. Indeed, if we consider an element $P = (P_1,\dots,P_n) \in \mathcal{P}^n$, then we have, by Proposition \ref{prop:Poincare} and since $\varphi([\mathcal{D}V](X)) = (0,\dots,0)$ by assumption,
\begin{align*}
|f_{V,X}(P)| &= |\langle [\mathcal{D}V](X), P(X) - \varphi(P(X))\rangle_{\varphi}|\\
&\leq \|[\mathcal{D}V](X)\|_2 \sqrt{\sum_i \|P_i(X) - \varphi(P_i(X))\|_2^2} \\
&\leq \|[\mathcal{D}V](X)\|_2\sqrt{C_{opt}(X) \sum_{i,j} \|\partial_j P_i(X)\|_2^2} \\
&= \sqrt{C_{opt}(X)}\|[\mathcal{D}V](X)\|_2 \times \|P\|_{H^1(\mu)},
\end{align*}
which establishes \eqref{eq:functional-bound} with $C = \sqrt{C_{opt}(X)}\|[\mathcal{D}V](X)\|_2$.

This implies existence of a Stein kernel $A = [\mathcal{J}Q_0](X)$, and moreover it satisfies $\|A\|_2^2 \leq C_{opt}(X)\|[\mathcal{D}V](X)\|_2^2$. Expanding the square in the definition of the discrepancy then yields
$$\Sigma^*(X|V)^2 \leq n + C_{opt}(X)\|[\mathcal{D}V](X)\|_2^2 - 2\langle [\mathcal{D}V](X), X\rangle_{\varphi}. $$
\end{proof}

Note that the construction is inherently different than the one from Section 2: while this one is the Jacobian of an element of $H^1(\mu)$, the first one is not. In particular, since projecting on the $L^2$-closure of the subspace $[\mathcal{J}(\mathcal{P}^n)](X) = \{[\mathcal{J} P](X) \mid P\in \mathcal{P}^n\}$ reduces the norm, the second one has the smallest possible norm, and hence gives rise to the infimum in the definition of the Stein discrepancy. More generally, two free Stein kernels for the same potential can differ only by an element in $[\mathcal{J}(\mathcal{P}^n)](X)^\bot \subset L^2(M_n(M \bar{\otimes} M^{op}, (\varphi \otimes \varphi^{op})\circ \Tr))$.

An interesting immediate corollary is the following reinforcement of Biane's characterization of the semicircular law. It is a free counterpart of a result of \cite{CFP17} on Gaussian distributions. 

\begin{cor}\label{cor:Poincare-characterization}
Let $X=(x_1,\dots,x_n)$ be an $n$-tuple of centered self-adjoint variables, and assume it satisfies $\sum^n_{i=1} \varphi(x_i^2) = n$. Then 
$$C_{opt}(X) \geq 1 + \frac{\Sigma^*\left(X| \frac{1}{2}\sum_{i=1}^n t_i^2\right)^2}{n}.$$
\end{cor}

Hence not only the semicircular law is the only isotropic distribution with Poincar\'e constant equal to one, but among all such distributions, if the Poincar\'e constant is close to one, then the variable must be close to a semicircular law.
In fact, as Corollary \ref{cor:Poincare-characterization} shows, this holds not only for isotropic distributions but more generally for distributions that come from $n$-tuples $X=(x_1,\dots,x_n)$ of centered self-adjoint variables satisfying $\sum^n_{i=1} \varphi(x_i^2) = n$. 

\vspace{3mm}

\underline{\textbf{ Acknowledgments }} : G. C. and M. F. were partly supported by the Project MESA (ANR-18-CE40-006) of the French National Research Agency (ANR). M. F. was also partly supported by Project EFI (ANR-17-CE40-0030) and ANR-11-LABX-0040-CIMI within the program ANR-11-IDEX-0002-02.

T. M. was supported by the ERC Advanced Grant NCDFP (339760) held by Roland Speicher.

\end{document}